\documentclass[12pt,twoside]{book}

\usepackage{mathpazo} 
\usepackage[euler-digits,small]{eulervm} 
\usepackage{bm} 

\newcommand*{\sheafhom}{\mathscr{H}\text{\kern -3pt {\calligra\large om}}\,}

\usepackage[shortlabels]{enumitem}
\usepackage{mathrsfs}
\usepackage{amsmath}
\usepackage{amssymb}
\usepackage{amsthm}
\usepackage{tikz-cd}
\usepackage{systeme}
\usepackage[margin=1in]{geometry}
\usepackage{url}
\usepackage{mathtools}
\usepackage{calligra}
\usepackage{multicol}
\usepackage{subfig}
\usepackage{float}
\usepackage[colorlinks = true,
linkcolor = blue,
urlcolor  = purple,
citecolor = blue,
anchorcolor = blue]{hyperref}

\newcommand{\iso}{\simeq}
\newcommand{\surj}{\twoheadrightarrow}
\newcommand{\inj}[1]{\xhookrightarrow{#1}}

\newcommand{\RR}{\boldsymbol{\mathbb{R}}}
\newcommand{\Rsheaf}{\mathcal{R}}

\theoremstyle{plain}
\newtheorem{theorem}{Theorem}[chapter]

\newtheorem{proposition}[theorem]{Proposition}
\newtheorem{lemma}[theorem]{Lemma}

\theoremstyle{definition}
\newtheorem{definition}[theorem]{Definition}

\newtheorem{example}[theorem]{Example}
\newtheorem{exercise}[theorem]{Exercise}

\numberwithin{equation}{theorem}

\usepackage{makeidx}
\makeindex

\begin{document}

\sloppy

\title{\Huge \bf Lecture Notes on the Fundamental Structures of Differential Geometry}
\author{Dmitrii Pedchenko \footnote{UCLA Mathematics Department,  Los Angeles, CA 90095-1555, \newline{} \hspace*{0.6 cm} {\it Email addresses:} dpedchenko@math.ucla.com, dmitry.pedchenko@gmail.com}}
\date{April 3, 2022}
\maketitle


\tableofcontents

\chapter{Introduction}
This is the first part of the lecture notes that grew out of the special course given by the author to the audience of exceptional undergraduate and first and second year graduate students at UCLA interested in geometry during the 2021/2022 academic year. 
In these lecture notes we present an approach to the fundamental structures  in differential geometry that uses the vernacular of sheaves, differential operators and horizontal subbundles. 

This language is now standard in other disciplines, notably algebraic and complex geometry,  but has been slow to enter the textbooks for students aspiring to work in differential geometry or use differential geometry in other geometric disciplines. 

A notable exception is the recent textbook \cite{wedhorn} that uses the language of sheaves and cohomology to introduce the basics of smooth manifolds. This textbook, however, does not tap into the differential geometry territory per se, as it does not discuss metrics, connections, curvature, parallel transport or holonomy groups. 

Another feature of this book is the treatment of the notion of connection from various different viewpoints: as a Koszul connection, as a covariant derivative, as an Ehresmann connection, as a parallel transport, etc. A similar extensive exposition of the notion of connection can be found in \cite{Spivak} and in \cite{postnikov}, but those expositions do not use the modern language that we try to develop in these notes. 

A byproduct of the chosen approach is a relatively simple and conceptually clear construction of canonical connections (Levi-Civita, Chern, Bismut, etc.) without a heavy reliance on explicit local expressions as a starting point.

Finally, these lecture notes were heavily inspired by the course \cite{verbit} in differential geometry given by Misha Verbitsky in 2013. 

\subsubsection*{Prerequisites}
This book assumes familiarity with the language of sheaves and the basics of smooth manifolds on the level of first year graduate courses. 

Both of these topics are covered in the already mentioned \cite{wedhorn}. For a more extensive modern treatment of the basics of smooth manifolds see the wonderful textbook \cite{JLee}. Other references include \cite{JohnLee}, \cite{postnikov} or the classic \cite{kobayashi1996foundations}. 

Familiarity with scheme theory and complex geometry will also be very helpful.

\subsubsection*{Acknowledgments}
First, we would like to thank Misha Verbitsky for a wonderful and broad education in geometry. The author took $5$ geometric courses and sat on numerous seminars organized by Misha while being an undergraduate student in the National Research University Higher School of Economics.

Second, I want to thank my scientific advisors Jack Huizenga and Alexei Pirkovskii for the ongoing support and interest in our work.

Last but not least, these lecture notes were prepared during the ongoing Russian invasion of Ukraine. Our heart is with those who had to bear the unnecessary and intolerable suffering one encounters during the war time. We also thank UCLA for the accommodations provided to the author during this rough period of life.

\chapter{Differential Operators}

\section{Differential operators over a ring}

\begin{definition}
A {\bf commutator} of two linear operators $A,B: V \to V$ is a linear operator defined as$$[A,B]:= A \circ B-B \circ A.$$ 
\end{definition}

The next key definition is due to A. Grothendieck.

\begin{definition}\label{def diff op} Let $R$ be a (commutative) $k$-algebra. A {\bf differential operator of order} $\mathbf{0}$ is a morphism of $R$-modules $D^0:R \to R$. Denote differential operators of order $0$ by $Diff^0(R)$. 

More generally, a $k$-linear morphism $D^i: R \to R$ is a {\bf differential operator of order} $\mathbf{i}$, $D^i \in Diff^i(R)$, if for any $D^0 \in Diff^0(R)$ the commutator $[D^i, D^0]$ lies in $Diff^{i-1}(R)$. 

We have $$Diff^0(R) \subset Diff^1(R) \subset Diff^2(R) \subset ...,$$
and the union $Diff^{\bullet}(R):=\cup_{i\geq0}Diff^i(R)$  is called the {\bf algebra of differential operators on} $\mathbf{R}$. 
\end{definition}

\begin{example}[Differential operators of order $0$]
It follows from the above definition that $$Diff^0(R) \iso R$$ via $D^0 \mapsto D^0(1)$. The inverse morphism is given by $a \mapsto (m_a: b \mapsto ab).$ We will consistently use this identification throughout this chapter. 
\end{example}

\begin{example}[Differential operators of order $1$] \label{derivations}
Let us show that $$Diff^1(R) \iso Der(R) \oplus R$$ as $k$-vector spaces. First note that for any derivation $X \in Der(R)$, any $m_a \in Diff^0(R)$ and any $b \in R$ we have $$[X, m_a](b) = X(ab) - aX(b) = aX(b) + bX(a) - aX(b) = X(a)b,$$ so we see that $[X, m_a]$ is just $m_{X(a)} \in Diff^0(R)$. Thus $Der(R) \subset Diff^1(R).$

Now consider $D^1 \in Diff^1(R)$. Define $$X(a) := D^1(a) - D^1(1) a, \quad \forall a \in R.$$

\begin{exercise}
Check that $X$ defined above is a derivation. 
\end{exercise}

This exercise shows that any $D^1 \in Diff^1(R)$ can be written as $$D_1 = X + m_{D^1(1)}.$$
\end{example}

Note that if $D^i \in Diff^i(R)$ and $D^j \in Diff^j(R)$, then for any $m_a \in Diff^0(R)$:
\begin{align*} Diff^{i+j-1} &\ni [m_a, D^i]\circ D^j + D^i \circ [m_a, D^j] 
 =m_a \circ D^i \circ D^j - D^i \circ m_a \circ D^j \\& + D^i \circ m_a \circ D^j - D^i \circ D^j \circ m_a = m_a \circ D^i \circ D^j - D^i \circ D^j \circ m_a = [m_a, D^i \circ D^j] \
\end{align*} 
which shows that $D^i \circ D^j \in Diff^{i+j}(R)$. This way, we get a {\bf filtered algebra} $\mathbf{Diff^{\bullet}(R)}$ of differential operators.

\begin{exercise}\label{ideals}
Let $D \in Diff^k(R)$. 
\begin{enumerate}[(a)]
\item Prove that $D(a_1a_2...a_k) = D'(a_2a_3...a_k)+a_1D(a_2a_3...a_k),$ where $D'=[D, m_{a_1}] \in Diff^{k-1}(R)$. 

\item Let $I \subset R$ be an ideal. Use induction to prove $D(I^{k+1}) \subset I.$
\end{enumerate}
\end{exercise}

\begin{exercise}\label{ass gr comm}
Use induction and the Jacobi identity $$[A, [B,C]]=[[A, B], C]+[B, [A,C]], \quad A,B,C \ - \text{linear operators},$$to prove that if $D^i \in Diff^i(R), D^j\in Diff^j(R)$, then the commutator $[D^i, D^j]$ lies in $Diff^{i+j-1}(R)$.
\end{exercise}

\subsection{Differential operators over the polynomial ring}

For this subsection, let $R:=\RR[t_1, ..., t_n]$.

\begin{lemma}\label{extension by zero}
If $D \in Diff^k(R)$ is a differential operator that vanishes on all polynomials of degree $\leq k$, then $D=0$.   
\end{lemma}

\begin{proof} Note that if $t_1^{i_1}\cdot ... \cdot t_n^{i_n}$ is a monomial of degree $|I|=i_1+...+i_n = k+1$, then by Exercise \ref{ideals} (a) $$D(t_1^{i_1}\cdot ... \cdot t_n^{i_n}) = D'(t^{|I|-1}) + t_iD(t^{|I|-1}), \quad D'=[D, m_{t_i}] \in Diff^{k-1}(R), \quad |I|-1=k.$$
The first term is zero by induction on the order of a differential operator and the fact that $D'$ vanishes on polynomials of degree $\leq k-1$. The second term is zero by the assumption of the lemma. 

The inductive step for monomials $t_1^{i_1}\cdot ... \cdot t_n^{i_n}$ of degree $|I|=i_1+...+i_n > k+1$ is handled using the same formula.
\end{proof}

\begin{lemma}\label{restriction}
Let $R_{\leq k} =:  \RR[t_1, ..., t_n]_{\leq k}$ be the space of polynomials of degree $\leq k$. Let $A: R_{\leq k} \to R$ be a linear operator. Then there exists $D \in Diff^k(R)$ such that $D|_{R_{\leq k}}: R_{\leq k} \to R$ coincides with $A$. 
\end{lemma}

\begin{proof}
Consider a differential operator $D_{f,i_1...i_n}=\frac{f(t_1, ..., t_n)}{i_1!\cdot ... \cdot i_n!}\left( \frac{d}{dt_1}\right)^{\circ i_1} \circ ... \circ \left( \frac{d}{dt_n} \right)^{\circ i_n}$ for $f \in R$. Note that $D_{f,i_1...i_n}$ is zero on all monomials of degree $<|I|=i_1+...+i_n$, $$D_{f,i_1 ... i_n} (t_1^{i_1}...t_n^{i_n}) =  f(t_1, ..., t_n),$$ and is zero on all other monomials of degree $|I|$. 

Thus $$D = \sum_{i_1+...+i_n \leq k}  D_{A(t_1^{i_n}...t_n^{i_n}),i_1...i_n}$$is the desired differential operator as monomials $A(t_1^{i_n}...t_n^{i_n})$ with $i_1+...+i_n \leq k$ form the basis of $R_{\leq k}$. 
\end{proof}

\begin{proposition}\label{poly generated}
The algebra of differential operators $Diff^{\bullet} \left(\RR[t_1, ..., t_n] \right)$ is generated by $t_i$ and $\frac{d}{dt_j}$ for $i,j=1,...,n$.
\end{proposition}
\begin{proof}
Let $D$ be a differential operator of order $k$. Consider $D|_{R_{\leq k}}: R_{\leq k} \to R$. By Lemma \ref{restriction}, there is $\tilde{D}\in Diff^k(R)$ that lies in the subalgebra generated by $t_i$ and $\frac{d}{dt_j}$ and that satisfies $D|_{R_{\leq k}} = \tilde{D}|_{R_{\leq k}}$. The differential operator $D - \tilde{D} \in Diff^k(R)$ vanishes on all polynomials of degree $\leq k$ by construction and therefore, by Lemma \ref{extension by zero}, $D = \tilde{D}$. 
\end{proof}

\begin{exercise}
Prove that $Diff^{\bullet}(R)$ is an $\RR$-algebra generated by $t_1, ..., t_n, \frac{d}{dt_1}, ..., \frac{d}{dt_n}$ with relations 
$$[t_i, t_j]=0, \quad \left[\frac{d}{dt_i}, \frac{d}{dt_j}\right]=0, \quad \forall i,j,$$
$$\left[t_i, \frac{d}{dt_i} \right]=1,$$
$$\left[ \frac{d}{dt_i}, t_j \right] = 0, \quad i \neq j.$$
\end{exercise}

\subsection{Differential operators over the ring of smooth functions}

In this subsection, we consider differential operators over the ring $C^{\infty}(\RR^n)$ of smooth function on $\RR^n$. We denote by $x_i$ the standard coordinate functions $x_i: (p_1,...p_n) \mapsto p_i.$

Let $m_x = \{ f \in R \ | f(x) =0 \}$ be the maximal ideal of smooth functions vanishing at point $x$. By Exercise \ref{ideals} (b), for any differential operator of order $i$ and $f \in m_x^{i+1}$ we have $D(f) =0$. This observation allows us to prove the following key lemma.

\begin{lemma}\label{lemma vanish}
Any differential operator $D \in Diff^i(C^{\infty}(\RR^n))$ that vanishes on polynomials of degree $\leq i$ is trivial. 
\end{lemma}

\begin{proof}
Consider $D \in Diff^i(C^{\infty}(\RR^n))$ that vanishes on $\RR[x_1, ..., x_n]_{\leq i}$. Take any $f \in C^{\infty}(\RR^n)$ and any point $x \in \RR^n$. Take the Taylor polynomial $P$ for function $f$ of degree $i$ at point $x$. We have $$f-P \in m_x^{i+1}.$$ By the observation above and the assumption that $D$ vanishes on $\RR[x_1, ..., x_n]_{\leq i}$ we conclude $$D(f-P) \in m_x \quad \implies \quad D(f-P) =0 \quad \implies \quad D(f) = D(P) + D(f-P) =0 .$$
\end{proof}

Next we prove an analogue of Proposition \ref{poly generated} for $Diff^{\bullet}(C^{\infty}(\RR^n)$. 
\begin{proposition}\label{classic}
The algebra $Diff^{\bullet}(C^{\infty}(\RR^n))$ is generated by $\frac{d}{dx_i}$ as a $C^{\infty}(\RR^n)$-algebra. 
\end{proposition}

\begin{proof} Similar to the proof of Lemma \ref{restriction}, given an $\RR$-linear map $$A: \RR[x_1, ..., x_n]_{\leq k} \to C^{\infty}(\RR^n)$$ we can explicitly construct a differential operator $$D \in C^{\infty}(\RR^n)\left<\frac{\partial}{\partial x_1}, ..., \frac{\partial}{\partial x_n}\right> \subset Diff^k(C^{\infty}(\RR^n))$$ whose restriction to $\RR[x_1, ..., x_n]_{\leq k}$ coincides with $A$ as a sum of differential operators
$$D = \sum_{i_1+...+i_n \leq k}  D_{A(x_1^{i_n}...x_n^{i_n}),i_1...i_n} \ ,$$
$$D_{f,i_1...i_n}=\frac{f(x_1, ..., x_n)}{i_1!\cdot ... \cdot i_n!}\left( \frac{\partial}{\partial x_1}\right)^{\circ i_1} \circ ... \circ \left( \frac{\partial}{\partial x_n} \right)^{\circ i_n}, \quad f \in C^{\infty}(\RR^n).$$

Now consider an arbitrary differential operator $D$ of order $k$. Consider the restriction $$D|_{\RR[x_1, ..., x_n]_{\leq k}} : \RR[x_1, ..., x_n]_{\leq k} \to C^{\infty}(\RR^n).$$ As above, construct a differential operator $\tilde{D}$ of order $i$ as a sum of $D_{f, i_1,..., i_n}$ with the property $D=\tilde{D}$ on $\RR[x_1, ..., x_n]_{\leq k}$. 

Consider $\overline{D} = D - \tilde{D} \in Diff^k(C^{\infty}(\RR^n)).$ By construction $\overline{D}$ vanishes on all polynomials of order $\leq k$. By Lemma \ref{lemma vanish} $\overline{D} = 0$ or, equivalently $$D = \tilde{D} \in C^{\infty}(\RR^n)\left<\frac{\partial}{\partial x_1}, ..., \frac{\partial}{\partial x_n}\right> \subset Diff^k(C^{\infty}(\RR^n)).$$
\end{proof}

This way,  in the case $R = C^{\infty}(\RR^n)$ Proposition \ref{classic} allows to relate Definition \ref{def diff op} to the historic analytical definition of a differential operator $D \in Diff^k(C^{\infty}(\RR^n))$ as an expression $$D= \sum_{i_1+...+i_n \leq k}  f_{i_1, ..., i_n}(x_1, ..., x_n) \frac{\partial^{i_1+...+i_n}}{\partial x_1^{i_1}... x_n^{i_n}}, \quad f_{i_1, ..., i_n}(x_1, ..., x_n) \in C^{\infty}(\RR^n).$$

The key Definition \ref{def diff op} can be sheafified if one replaces a ring $R$ with a \emph{sheaf} of algebras $\mathcal{R}$ on a topological space $X$. 

\begin{definition}\label{sheaf def} Let $\mathcal{R}$ be a sheaf of (commutative) $k$-algebras on a topological space $X$. A {\bf sheaf of differential operators on} $\bm{\mathcal{R}}$ is a functor
$$\{U \subset X \ \text{open} \}^{op} \to \text{Filtered $k$-algebras},$$
$$U \mapsto Diff^{\bullet}(\mathcal{R}(U)),$$
where $Diff^{\bullet}(\mathcal{R}(U))$ is the filtered algebra of differential operators over $\mathcal{R}(U)$ in the sense of Definition \ref{def diff op}. One readily checks that this assignment satisfies the sheaf axioms. 

We will denote this sheaf by $Diff^{\bullet}_{\mathcal{R}}$ or simply by $Diff^{\bullet}$ if the underlying sheaf $\mathcal{R}$ is clear from the context. Observe that after forgetting the filtration the sheaf of differential operators is a subsheaf of the sheaf of internal homomorphisms 
$$ Diff^{\bullet}_{\mathcal{R}} \subset \sheafhom_{k}(\mathcal{R}, \mathcal{R}),$$ where homomorphisms are taken with respect to the $k$-vector space structure of $\mathcal{R}$.

Finally, if $D \in Diff^i_{\mathcal{R}}(U) = Diff^i(\mathcal{R}(U))$, then we say that $D$ is a {\bf differential operator of order $\bm{i}$ on $\bm{U}$}. \end{definition}

\section{Algebra of symbols of differential operators over a ring}\label{symbols section}

Given a $k$-algebra $R$ or a sheaf of $k$-algebras $\Rsheaf$, consider the filtered algebra of differential operators $$Diff^0(R) \subset Diff^1(R) \subset Diff^2(R) \subset ... \subset Diff^{\bullet}(R)$$and the filtered sheaf of algebras of differential operators $$Diff^0_{\Rsheaf} \subset Diff^1_{\Rsheaf} \subset Diff^2_{\Rsheaf} \subset ... \subset Diff^{\bullet}_{\Rsheaf}.$$

\begin{proposition}\label{commutative}
The associated graded algebra $\oplus_{i \geq 0} Diff^i(R)/Diff^{i-1}(R)$ and the associated graded sheaf of algebras $\oplus_{i \geq 0} Diff^i_{\Rsheaf}/Diff^{i-1}_{\Rsheaf}$ are commutative. 
\end{proposition}
\begin{proof}
By Exercise \ref{ass gr comm}, for $D^i \in Diff^i(R), D^j \in Diff^j(R)$, we have $$[D^i, D^j] = D^i \circ D^j - D^j \circ D^i \in Diff^{i+j-1}(R).$$Hence as elements of $Diff^{i+j}(R)$ $$D^i\circ D^j \equiv D^j \circ D^i \quad \mod Diff^{i+j-1}(R).$$
\end{proof}

\begin{definition}
The associated graded algebra $$ Symb^* (Diff^{\bullet}(R)):=\oplus_{i \geq 0} Symb^i :=gr(Diff^{\bullet}(R)) = \oplus_{i \geq 0}Diff^i(R)/Diff^{i-1}(R)$$ is called the {\bf algebra of symbols of differential operators over} $\bm{R}$. Note that $S^0 = Diff^0(R)$ is canonically isomorphic to $R$, so $Symb (Diff^{\bullet}(R))$ is a (commutative) $R$-algebra.

We can similarly define a {\bf sheaf of (commutative) $\bm{\mathcal{R}}$-algebras of symbols of differential operators over $\bm{\mathcal{R}}$}:  $$ Symb^* (Diff^{\bullet}_\Rsheaf):=\oplus_{i \geq 0} Symb^i :=gr(Diff^{\bullet}_\Rsheaf) = \oplus_{i \geq 0}Diff^{i}_\Rsheaf/Diff^{i-1}_\Rsheaf.$$
\end{definition}

Here

\begin{example}\label{symb 1}
Let us show that $Symb^1 = Diff^1(R)/Diff^0(R)$ is isomorphic to the $R$-module $Der(R)$ of derivations of $R$. 
\end{example}

\begin{proof}
In Example \ref{derivations} we showed that every derivation is also a differential operator of order $1$. This way we get a morphism $$Der(R) \inj{} Diff^1(R) \surj Diff^1(R)/Diff^0(R).$$

The inverse morphism is defined by $$  Diff^1(R) \ni D  \to D - m_{D(1)} \in Der(R).$$ One easily checks that $D-m_{D(1)}$ is indeed a derivation. Note that for $D=m_a \in  Diff^0(R)$ the difference $D - m_{D(1)} = m_a - m_{a \cdot 1} = 0$, so the inverse morphism is in fact defined as a map $$Diff^1(R)/ Diff^0(R) \to Der(R).$$  
\end{proof}

\begin{proposition}\label{key homo}
There is a natural homomorphism of graded commutative $R$-algebras $$Sym^*_R(Der(R)) = \oplus_{i \geq 0}Sym^i_R(Der(R)) \to \oplus_{i \geq 0} Symb^i = Symb^*(Diff^{\bullet}(R))$$that extends the isomorphism $Der(R) \stackrel{\iso}{\to} Symb^1$, 
and an analogous natural homomorphism of sheaves of graded commutative $\Rsheaf$-algebras
$$Sym^*(Der_{\Rsheaf}) \to Symb^*(Diff^{\bullet}_{\Rsheaf}).$$
\end{proposition}

\begin{proof}
Consider the following diagram of solid arrows:

\begin{center}\begin{tikzcd}
Sym^*(Der(R)) \arrow[r, dashed]  &  Symb^*(Diff^{\bullet}(R))  \\
Der(R) \arrow[u, hook] \arrow[r, "\iso"]          & Symb^1 \arrow[u, hook].
\end{tikzcd} \end{center}
By Proposition \ref{commutative} the target algebra  $Symb^*(Diff^{\bullet}(R))$ is commutative.
The upper horizontal dashed arrow is then just the  universal property of the symmetric algebra applied to the bottom isomorphism of Example \ref{symb 1} followed by the inclusion ${Symb^1 \inj{} Symb^*(Diff^{\bullet}(R))}$ on the right of the commutative diagram. 
\end{proof}

\section{Algebra of symbols of differential operators over the ring of smooth functions on a manifold}

In this section we will specialize the results of \S \ref{symbols section} on the symbols of differential operators to the setting of differential operators on a smooth manifold $M$. 

Consider the following objects:
\begin{itemize}
\item the ring $R=C^{\infty}(M)$ of smooth functions on a smooth manifold $M$, 
\item the sheaf $\Rsheaf=C^{\infty}_M$ of smooth functions on $M$ (so that $R=C^{\infty}(M)=\Gamma(M, C^{\infty}_M)=\Gamma(M, \Rsheaf)$), 
\item the algebra $Diff^{\bullet}(M):=Diff^{\bullet}(C^{\infty}(M))$ of global differential operators on $M$, 

\item the sheaf of algebras $Diff^{\bullet}_M:=Diff^{\bullet}_{C^{\infty}_M}$ of differential operators on $M$ in the sense of  Definition \ref{sheaf def} (so that ${Diff^{\bullet}(M) = \Gamma(M, Diff^{\bullet}_M))}$.
\end{itemize}

While the similarity of notation for the global and local objects might appear slightly confusing at first, the reader will soon find that working with sheafified objects is almost identical to working with global objects. This justifies our choice of notation.

The well known equivalence between vector fields and derivations of the sheaf of smooth functions $C^{\infty}_M$ yields an isomorphism of sheaves
$$Der(C^{\infty}_M) \iso TM,$$where we identify the tangent bundle $TM$ with the locally free sheaf of its sections (vector fields). In this setting, the morphism of Proposition \ref{key homo} becomes $$Sym^*(TM) \to Symb^*(Diff^{\bullet}_M).$$

The first goal of this section will be to prove the following key theorem.

\begin{theorem} Let $M$ be a smooth manifold. The natural homomorphism of sheaves
$$Sym^*(TM) \to Symb^*(Diff^{\bullet}_M)$$ is an isomorphism.
\end{theorem}

An immediate corollary of this theorem is that the sheaf of symbols of differential operators on a manifold is \emph{locally free}. We will prove the above theorem as a series of lemmas.

\newpage
\addcontentsline{toc}{chapter}{Bibliography}
\bibliographystyle{alpha}

\end{document}